\documentclass[paper=a4,english,fontsize=11pt,parskip=half,abstract=true]{scrartcl}
\usepackage{babel}
\usepackage[utf8]{inputenc}
\usepackage[T1]{fontenc}
\usepackage[left=20mm,right=20mm,top=30mm,bottom=30mm]{geometry}
\usepackage{amsmath}
\usepackage{amsthm}
\usepackage{amssymb}
\usepackage{enumerate} 
\usepackage{thmtools}
\usepackage{mathtools}
\mathtoolsset{centercolon} 
\usepackage{stmaryrd} 
\usepackage[bookmarks=true,
            pdftitle={A solution to Brauer's Problem 14},
            pdfauthor={Benjamin Sambale},
            pdfkeywords={},
            pdfstartview={FitH}]{hyperref}

\newtheorem{Thm}{Theorem} 
\newtheorem*{ThmA}{Theorem~A} 
\newtheorem{Lem}[Thm]{Lemma}

\theoremstyle{definition}

\numberwithin{equation}{section}

\allowdisplaybreaks[1]

\renewcommand{\phi}{\varphi}
\newcommand{\C}{\mathrm{C}}

\newcommand{\Z}{\mathrm{Z}}

\newcommand{\ZZ}{\mathbb{Z}}
\newcommand{\CC}{\mathbb{C}}

\newcommand{\RR}{\mathbb{R}}
\newcommand{\NN}{\mathbb{N}}

\newcommand{\PSU}{\operatorname{PSU}}

\newcommand{\Irr}{\operatorname{Irr}}

\title{A solution to Brauer's Problem 14}
\author{John Murray\footnote{Department of Mathematics and Statistics, National University of Ireland, Maynooth, County Kildare, Ireland, \href{mailto:John.Murray@mu.ie}{John.Murray@mu.ie}} \ and Benjamin Sambale\footnote{Institut für Algebra, Zahlentheorie und Diskrete Mathematik, Leibniz Universität Hannover, Welfengarten 1, 30167 Hannover, Germany,
\href{mailto:sambale@math.uni-hannover.de}{sambale@math.uni-hannover.de}}}
\date{\today}

\begin{document}
\frenchspacing
\maketitle
\begin{abstract}\noindent
It is well known that the number of real irreducible characters of a finite group $G$ coincides with the number of real conjugacy classes of $G$. Richard Brauer has asked if the number of irreducible characters with Frobenius--Schur indicator $1$ can also be expressed in group theoretical terms. We show that this can done by counting solutions of $g_1^2\ldots g_n^2=1$ with $g_1,\ldots,g_n\in G$. 
\end{abstract}


\section{Introduction}

Since Brauer's survey~\cite{BrauerLectures}, there has been some interest in expressing representation theoretical invariants of a finite group $G$ in terms of group theoretical descriptions. 
The most basic observation of this kind is probably that the number of irreducible characters of $G$ equals the number of conjugacy classes of $G$. 
In principle the whole character table of $G$ is implicitly determined via the class multiplication constants
\[|\{(x,y)\in C\times D:xy=z\}|,\]
where $C,D,E$ are conjugacy classes and $z\in E$ is fixed (this is the basis of the Dixon--Schneider algorithm; see \cite[Corollary~2.4.3]{LuxPahlings}). It is often desirable to have more direct relations. 
Brauer was particularly interested in the existence of $p$-\emph{defect zero} characters, i.\,e. $\chi\in\Irr(G)$ such that the degree $\chi(1)$ is divisible by the $p$-part $|G|_p$. Strunkov~\cite{Strunkov} (see also \cite[Theorem~4.12]{Navarro2}) showed that such characters exist if and only if there exists $g\in G$ such that $p$ does not divide
\[|\{(x,y)\in G^2:[x,y]=g\}|,\]
where $[x,y]=xyx^{-1}y^{-1}$ denotes the commutator of $x$ and $y$. 
Similar criteria were obtained by Barker~\cite{BarkerDefect}, Broué~\cite{BroueRobinson}, Qian~\cite{Qian} and Shi~\cite{Shi}. 
Robinson~\cite{RobinsonDef0} (see also \cite[Theorem~4.20]{Navarro}) has expressed the precise number of $p$-defect zero characters as the rank of a certain matrix defined in group theoretical terms. This has answered Brauer's Problem~19 of \cite{BrauerLectures}.
In a different paper, Robinson~\cite{RobinsonComm} has shown that the multiset of irreducible character degrees of $G$ is determined by the group theoretical sequence
\[|\{(a_1,b_1,\ldots,a_n,b_n)\in G^{2n}:[a_1,b_1]\ldots[a_n,b_n]=1\}|\qquad(n\in\NN).\]

Brauer's permutation lemma on the character table implies that the number of real irreducible characters coincides with the number of real conjugacy classes. Here, a conjugacy class $C$ of $G$ is called \emph{real} if $C=\{x^{-1}:x\in C\}$. Robinson has further shown that the number of real characters of a given degree is computable in group theoretical terms. Brauer's Problem~14 asks whether one can describe the number of characters which arise from real irreducible \emph{representations} (i.\,e. $\chi\in\Irr(G)$ with Frobenius--Schur indicator $\epsilon(\chi)=1$) group theoretically. Note that this number is not encoded in the character table. We answer this question positively as follows.

\begin{ThmA}
Let $G$ be a finite group with $k_r(G)=|\Irr_\RR(G)|$ real conjugacy classes. Then the multiset $\{\chi(1)\epsilon(\chi):\chi\in\Irr_\RR(G)\}$ is determined by the sequence
\[s(n):=|\{(g_1,\ldots,g_n)\in G^n:g_1^2\ldots g_n^2=1\}|\]
for $n=1,\ldots,k_r(G)+1$. In particular, $k_r(G)=\frac{s(2)}{|G|}$ and the number of irreducible characters of $G$ with Frobenius--Schur indicator $1$ can be described purely in group theoretical terms.
\end{ThmA}

\section{Proofs}

We start with a combinatorial lemma.

\begin{Lem}\label{lem}
Let $a_1,\ldots,a_n\in\CC$. Then the multiset $\{a_1,\ldots,a_n\}$ is uniquely determined by the power sums $\sum_{i=1}^na_i^k$ for $k=0,1,\ldots,n$.
\end{Lem}
\begin{proof}
Let $\sigma_0:=1,\sigma_1,\ldots,\sigma_n$ be the elementary symmetric functions in $n$ variables. By the Girard--Newton identities (see \cite[Theorem~8.7]{SambalePS}), the values $\sigma_k(a_1,\ldots,a_n)$ can be computed from the power sums. Hence, the polynomial \[(X-a_1)\ldots(X-a_n)=\sum_{k=0}^n(-1)^{n-k}\sigma_{n-k}(a_1,\ldots,a_n)X^k\] 
is uniquely determined and so are its roots.
\end{proof}

Remark:
Suppose that $a_1,\ldots,a_n$ are non-zero. Let $\rho_k(a_1,\ldots,a_n):=\sum_{i=1}^na_i^k$ for $k\in\ZZ$. With the notation of the proof above, we have
\[\sigma_n(a_1,\ldots,a_n)\rho_{-1}(a_1,\ldots,a_n)=a_1\ldots a_n\sum_{i=1}^na_i^{-1}=\sigma_{n-1}(a_1,\ldots,a_n).\]
Hence, if $\rho_{-1}(a_1,\ldots,a_n)$ is known and non-zero, then $\rho_n(a_1,\ldots,a_n)$ is not required to compute $a_1,\ldots,a_n$. This will be used in the following proof.

Recall that the Frobenius--Schur indicator of $\chi\in\Irr(G)$ is defined by 
\begin{equation}\label{defFS}
\epsilon(\chi):=\frac{1}{|G|}\sum_{g\in G}\chi(g^2)\in\{0,1,-1\}.
\end{equation}

\begin{proof}[Proof of Theorem~A]
Consider $S:=\frac{1}{|G|}\sum_{g\in G}g^2\in\Z(\CC G)$. For a fixed $n\in\NN$ we may write $S^n=\sum_{g\in G}\alpha_gg$ with $\alpha_g\in\ZZ$ for $g\in G$. Note that $s(n)=\alpha_1|G|^n$. Let $\chi\in\Irr(G)$. Recall that the function $\frac{\chi}{\chi(1)}$ extends to an algebra homomorphism $\omega_\chi\colon\Z(\CC G)\to\CC$. By \eqref{defFS}, $\omega_\chi(S)=\frac{\epsilon(\chi)}{\chi(1)}$. 
By the orthogonality relation,
\begin{align*}
\frac{s(n)}{|G|^{n-1}}&=\sum_{g\in G}\alpha_g\sum_{\chi\in\Irr(G)}\chi(g)\chi(1)=\sum_{\chi\in\Irr(G)}\chi(1)^2\sum_{g\in G}a_g\frac{\chi(g)}{\chi(1)}\\
&=\sum_{\chi\in\Irr(G)}\chi(1)^2\omega_\chi(S)^n=\sum_{\chi\in\Irr(G)}\frac{\epsilon(\chi)^n}{\chi(1)^{n-2}}=\sum_{\chi\in\Irr_\RR(G)}\Bigl(\frac{\epsilon(\chi)}{\chi(1)}\Bigr)^{n-2}
\end{align*}
for every $n\in\NN$. In particular, $s(2)=|G|k_r(G)$. We apply \autoref{lem} with the non-zero numbers $\{a_1,\ldots,a_{k_r(G)}\}=\{\frac{\epsilon(\chi)}{\chi(1)}:\chi\in\Irr_\RR(G)\}$. The power sums are given in terms of the $s(n)$. 
By the remark after \autoref{lem}, the multiset $\{\chi(1)\epsilon(\chi):\chi\in\Irr_\RR(G)\}$ is determined by $s(n)$ for $n=1,\ldots,k_r(G)+1$ since $s(1)>0$. 
\end{proof}

If all $\chi\in\Irr_\RR(G)$ have $\epsilon(\chi)=1$, then the proof shows that $\frac{s(n)}{|G|^n}$ is a non-increasing sequence. 
Hence, if there exists some (odd) $n$ such that $s(n)|G|< s(n+1)$, then some $\chi\in\Irr(G)$ has $\epsilon(\chi)=-1$. 
This criterion applies to the simple group $G=\PSU(3,3)$ with $n=5$ as one can check with GAP~\cite{GAPnew}. It does however not apply to the McLaughlin group $McL$, which is the only sporadic group with some $\epsilon(\chi)=-1$.
In general it is easy to show  that 
\[\lim_{n\to\infty}\frac{s(n)}{|G|^{n-1}}=|G/N|\] 
where $N:=\langle g^2:g\in G\rangle$ is the smallest normal subgroup of $G$ with elementary abelian $2$-quotient.

Since $k_r(G)$ and $s(2)$ are group theoretical invariants, one may ask if the equality $k_r(G)|G|=s(2)$ from Theorem~A can be proved without characters. To this end, let $C$ be a real conjugacy class and $x\in C$. Then there exist exactly $|\C_G(x)|$ elements $y\in G$ such that $y^{-1}xy=x^{-1}$. Now there are $|C||\C_G(x)|=|G|$ pairs $(x,y)\in G^2$ such that $y^{-1}xy=x^{-1}$ and $x\in C$. Consequently, it suffices to verify that the maps
\begin{align*}
\{(x,y)\in G^2:y^{-1}xy=x^{-1}\}&\longleftrightarrow\{(g,h)\in G^2:g^2h^2=1\},\\
(x,y)&\longmapsto(xy^{-1},y),\\
(gh,h)&\longmapsfrom(g,h)
\end{align*}
are mutually inverse bijections (we leave this to the reader). 

A straight-forward adaptation of Theorem~A shows that the higher Frobenius--Schur indicators 
\[\epsilon_k(\chi):=\frac{1}{|G|}\sum_{g\in G}\chi(g^k)\qquad(k\in\NN,\ \chi\in\Irr(G))\]
can be determined similarly in group theoretical terms.

\section*{Acknowledgment}
The paper was written while the first author visited the University of Hannover in November 2022. Both authors thank Burkhard Külshammer for some interesting discussion on the subject. Also Rod Gow made some valuable suggestions to improve the paper.
The second author is supported by the German Research Foundation (\mbox{SA 2864/4-1}).

\end{document}